\documentclass[a4paper, 10pt]{amsart}

\usepackage{amssymb,amscd,stmaryrd, bm}
\usepackage[mathcal]{eucal}
\usepackage{array,float}
\usepackage{enumitem}
\usepackage{mathtools}
\usepackage{xcolor}

\usepackage{xy}
\input xy
\xyoption{all}
\usepackage{pdflscape}
\usepackage{hyperref}

\numberwithin{equation}{section}
\numberwithin{equation}{subsection}

\newtheorem{thm}{Theorem}[section]
\newtheorem{proposition}[thm]{Proposition}

\newtheorem{corollary}[thm]{Corollary}
\newtheorem{lemma}[thm]{Lemma}

\newtheorem{definition}[thm]{Definition}
\newtheorem*{remark*}{Remark}

\newcommand{\ind}{{\text{Ind}}}
\newcommand{\Hom}{{\mathrm{Hom}}}

\newcommand{\tra}{{\mathrm{tra}}}
\newcommand{\res}{{\mathrm{res}}}
\newcommand{\PGL}{{\mathrm{PGL}}}
\newcommand{\Ho}{{\mathrm{H}}}
\newcommand{\GL}{\mathrm{GL}}

\newcommand{\bZ}{{\mathbb Z}}
\newcommand{\bC}{{\mathbb C}}
\newcommand{\pooja}[1]{{\color{blue} #1}}

\newcommand{\irr}{\mathrm{Irr}}

\begin{document}
\title[Projective representations of Heisenberg groups]{On Schur multiplier and Projective \\ representations of 
Heisenberg  groups}
%\date{\today }

\author{Sumana Hatui}
\address{SH: Department of Mathematics,
	Indian Institute of Science,
	Bangalore 560012, India }

\email{sumanahatui@iisc.ac.in}
\author{Pooja Singla}
\address{PS: Department of Mathematics and Statistics, Indian Institute of Technology Kanpur, Kanpur 208016, INDIA}

\email{psingla@iitk.ac.in}

\begin{abstract}

In this article, we describe the Schur multiplier and representation group of discrete Heisenberg groups and their $t$-variants. We give a construction of all complex finite-dimensional irreducible projective representations of these groups.

\end{abstract}


\maketitle 

\section{Introduction}

The theory of projective representations involves understanding homomorphisms from a group into the projective linear groups. Schur~\cite{IS4, IS7, IS11} extensively studied it. These representations appear naturally in the study of ordinary representations of groups and are known to have many applications in other areas of Physics and Mathematics.  We refer the reader to Section~\ref{se2} for precise definitions and related results regarding projective representations of a group. By definition, every ordinary representation of a group is projective, but the converse is not true. 
Therefore, understanding the projective representations is usually more intricate.
 Recall, the Schur multiplier of a group $G$ is the second cohomology group $\Ho^2(G, \mathbb{C}^{\times})$, where $\bC^\times$ is a trivial $G$-module. The Schur multiplier of a group plays an important role in understanding its projective representations.
By definition, every projective representation $\rho$ of $G$ is associated with a 2-cocycle $\alpha: G \times G \to \bC^\times$ such that $\rho(x)\rho(y)=\alpha(x,y)\rho(xy)$ for all $x,y \in G$. In this case, we say, $\rho$ is an $\alpha$-representation. Conversely, for every 2-cocycle $\alpha$ of $G$, there exists an $\alpha$-representation of $G$, namely $\mathbb C^\alpha(G)$ the twisted group algebra of $G$. So, the first step towards understanding the projective representations is to describe the 2-cocycles of $G$ up to cohomologous, i.e., to understand the Schur multiplier of $G$. The second step involves constructing $\alpha$-representations of $G$ for all $[\alpha] \in \Ho^2(G, \mathbb C^\times)$, where $[\alpha]$ denotes the cohomology class of $\alpha$.

The complex ordinary representations of finite abelian groups are easy to understand.
 For example, all irreducibles are one dimensional. But this is not true for
their projective representations. This problem has been studied by many authors, most notably by Morris, Saeed-ul-Islam, and Thomas in \cite{Moa}, \cite{Moa2}. All irreducible $\alpha$-representations of $(\bZ/n\bZ)^k$ for some special $\alpha$ have been described in \cite{Moa}. This work was generalized to all finite abelian groups for some special class of cocycles in \cite{Moa2}. Their results are outlined in \cite[Chapter 3]{GK1} and \cite[Chapter 8]{GK3}. Later, Higgs \cite{RJ} constructed  an irreducible $\alpha$-representation of elementary abelian $p$-groups $(\bZ/p\bZ)^k$, for every $\alpha$. Also, he counted the number of $[\alpha] \in \Ho^2((\bZ/p\bZ)^k, \bC^\times)$ such that irreducible $\alpha$-representations of $(\bZ/p\bZ)^k$ continue to be irreducible when restricted to a subgroup of index $\leq p^2$. The corresponding results for $(\bZ/p^r\bZ)^k$ with $r>1$ are not yet known. The projective representations of dihedral groups are also well known in the literature; see ~\cite[Theorem~7.3]{GK1}. Schur~\cite {IS11} studied the projective representations of the symmetric groups $S_n$. He proved that the Schur multiplier of $S_n$ for, $n \geq 4$, is $\bZ/2\bZ$ and described the representation group of $S_n$, see~\cite{Moa1, WB} for more details. Nazarov~\cite{Na1, Na2} explicitly constructed the projective representations of $S_n$ by providing suitable orthogonal matrices for each generator of the symmetric group. 

In this article, our goal is to describe the Schur multiplier and the projective representations of the discrete Heisenberg groups and their $t$-variants. The $t$-variants of the Heisenberg groups, denoted by $H^t_{2n+1} (R)$, are defined as follows. Let $R$ be a commutative ring with identity and $t \in R$.  Define the group $H^t_{2n+1} (R)$ by the set $R^{n+1} \oplus  R^n$ with multiplication given by, 
\[
\begin{array}{l} 
(a, b_1, \ldots, b_n, c_1, \ldots, c_n) (a', b'_1, b'_2, \ldots, b'_n, c'_1, c'_2, \ldots, c'_n) \\ 
= (a+a'+ t(\sum_{i = 1}^n b'_i c_i), b_1 + b'_1, \ldots, b_n + b'_n, c_1 + c'_1, \ldots, c_n + c'_n).
\end{array}
\]
For $t = 1$, we recover the classical Heisenberg group and throughout we denote $H^1_{2n+1} (R)$ by $H_{2n+1} (R)$. 
Except Theorem~\ref{thm2}, which is true for general commutative rings $R$ with identity, the ring $R$ will be $ \bZ/r\bZ$ for $r \in \mathbb N \cup \{0\}$. It follows from \cite [Corollary 5.1.3]{GK1} that the projective representations of  $H^t_{2n+1}(\bZ/r\bZ)$ are obtained from those of $H^t_{2n+1}(\bZ/p_i^{m_i}\bZ)$, where $r=p_1^{m_1}p_2^{m_2}\cdots p_k^{m_k}$ is the prime decomposition of $r$. 
Hence, for $r \in \mathbb N$, we can further assume that $t|r$.

Our first result describes the Schur multiplier of  $H^t_{2n+1} (R)$ for $R = \bZ/r\bZ$. 
The description of the Schur multiplier of $H_{2n+1}^t(R)$ for $n>1$ differs from the case $n=1$. For $n = 1$, we further assume that either $r =0 $ or $r$ is an odd natural number.

\begin{thm} \label{thm1}
	\begin{enumerate}[label=(\roman*)] 
				\item  For $n>1$,
		\[
		\Ho^2({H}_{2n+1}^t(\mathbb Z/r\bZ),\bC^\times)= 
		\begin{cases}
		(\bZ/r\bZ)^{2n^2-n-1}  \times (\bZ/t\bZ)^{2n+1},& \text{if }r \in \mathbb N,\\
		(\bC^\times)^{2n^2-n-1}  \times (\bZ/t\bZ)^{2n},              & \text{if } r=0.
		\end{cases}
		\]
		\item For $r \in (2 \mathbb N + 1) \cup \{0\}$, 
		\[
		\Ho^2({H}_{3}^t(\mathbb Z/r\bZ),\bC^\times)= 
		\begin{cases}
		(\bZ/r\bZ )^2 \times \bZ/t\bZ,& \text{if }r \in (2 \mathbb  N +1),\\
		(\bC^\times)^2,              & \text{if }  r=0.
		\end{cases}
		\]

	\end{enumerate}
\end{thm}
The Schur multiplier of $H_3(\bZ/r\bZ)$ was obtained in \cite[Theorem 1.1]{UJ}. A proof of the above result is included in Section~\ref{se1}. 

Our next aim is to describe the projective representations of  $H_{2n+1}^t(\bZ/r\bZ)$. Throughout this article, we consider these groups as discrete (abstract) groups and therefore the obtained projective representations  may not be unitary or even continuous. It is well known that the projective representations of a group $G$ are obtained from the ordinary representations of its representation group (if it exist); see Corollary \ref{rep_group}.  Our next result describes a representation group of $H_3^t(\bZ/r\bZ)$. For $r \in \mathbb N \cup \{ 0 \} $, define a group $\widehat{H}(r,t)$ by 
\[
\widehat{H}(r,t)= \langle x, y, z\mid  [x, y ] =z^t, [x, z] = z_1,  [y,z] =z_2, x^{r} = y^{r} = z^{rt}=1 \rangle.\]
Throughout the article, $[x,y] = xyx^{-1}y^{-1}$ and the relations of the form $[x, y]= 1$  for generators $x$ and  $y$ are omitted in the presentation of a group. 
\begin{thm}\label{thm4}
	For $r \in (2\mathbb N+1) \cup \{ 0 \} $ and $t\mid r$, the group  $\widehat{H}(r,t)$ is a representation group of ${H}_3^t(\mathbb Z/r\bZ)$. 
\end{thm}
See Section~\ref{se2} for the proof of this result. A construction of all finite-dimensional irreducible ordinary representations of $\widehat{H}(r,t)$ is included in Section~\ref{se3}.  Our next result focuses on the projective representations of $H_{2n+1}^t(R)$ for $n> 1$. Recall that the group $H_{2n+1}^t(R)$ projects onto the abelian group $R^{2n}\oplus R/tR$ (see (\ref{eq:Heisenberg-stem})). The following result is true for general commutative rings $R$ with identity. 
\begin{thm}
	\label{thm2}
For $n > 1$, every irreducible projective representation of $H_{2n+1}^t(R)$ is obtained from an irreducible projective representation of the abelian group $R^{2n}\oplus R/tR$ via inflation.
\end{thm}
We obtain its proof from a general result regarding the central product of groups; see Corollary~\ref{HG} and Section~\ref{subsec:central-product-proof}. 
From the above result, the question of determining the projective representations of $H_{2n+1}^t(R)$ for $n> 1$ boils down to understanding the projective representations of abelian groups $R^{2n}\oplus R/tR$. As mentioned earlier, this result is not yet well understood.  
 Next, for $R = \bZ/r\bZ$ and $n \in \mathbb N$, we describe the representation group of $R^{n}\oplus R/tR$. Define the group $\mathcal{F}_n(r, t)$ as follows.
\[
\mathcal{F}_{n}(r,t) = \langle x_k, z_{ij} \mid 1 \leq k \leq n+1, 1 \leq i < j \leq n+1, [x_i, x_j ] = z_{ij},x_1^t= x_j^r = 1 \rangle.
\]

\begin{thm} \label{thm3} 
	For $r \in \mathbb N \cup \{0\}$ and $t \mid r$, 
	the group  $\mathcal{F}_{n}(r,t)$ is a representation group of $(\mathbb Z/r\mathbb Z)^{n}\oplus \bZ/t\bZ$.
\end{thm}
A proof of this result is included in Section~\ref{se2}. See Section~\ref{se3}, for a construction of all finite-dimensional ordinary irreducible representations of $\mathcal{F}_{n}(r,t)$. 
We also obtain results regarding the projective representations of extra-special groups. 
Recall that a $p$-group $G$ is called an extra-special group if
its center $Z(G)$ is cyclic of order $p$ and the quotient $G/Z(G)$ is a non-trivial elementary abelian $p$-group. It is well known that  for each $n \geq 1$, there are two extra-special $p$ groups of order $p^{2n+1}$ up to isomorphism with exponents either $p$ or $p^2$. We denote the isomorphism classes of extra special groups of order $p^{2n+1}$ with exponent $p$ and $p^2$ by $ES_{2n+1}(p)$ and $ES_{2n+1}(p^2)$ respectively. 
From definition, the groups $ES_{2n+1}(p)$ are  isomorphic to $H_{2n+1}(\bZ/p\bZ)$. Above, we have already stated the results regarding the projective representations for $H_{2n+1}(\bZ/p\bZ)$. Combining this with  our next result, we complete the picture for extra-special $p$-groups.

\begin{corollary}\label{thm5}
	(i)  Every projective representation of $ES_{3}(p^2)$ is equivalent to an ordinary representation. 
	
	(ii) For $n>1$,  every irreducible projective representation of $ES_{2n+1}(p^2)$  is obtained from an irreducible projective representation of $(\bZ/p\bZ)^{2n}$ via inflation.
\end{corollary}
Above, (i) follows because the Schur multiplier of $ES_{3}(p^2)$ is trivial; see \cite[Theorem 3.3.6]{GK}. For the proof of (ii), see Section~\ref{subsec:central-product-proof}.

\section{Schur multiplier of $H_{2n+1}^t(\bZ/r\bZ), r \in \mathbb N \cup \{0\}$} \label{se1}
In this section, we prove Theorem~\ref{thm1}. Throughout this article, we use $x^y$ to denote the conjugation $yxy^{-1}$. 
The commutator subgroup and center of a group $G$ are denoted by  $G'$ and $Z(G)$, respectively.

Recall, for a group $G$ and $i \in \mathbb N$, $\Ho^i(G, \mathbb C^\times) = Z^i(G, \mathbb C^\times)/B^i(G, \mathbb C^\times)$, where $Z^i(G, \mathbb C^\times)$ and $B^i(G, \mathbb C^\times)$ consists of cocycles and coboundaries of $G^i$ respectively. We shall call  elements of $Z^2(G, \mathbb C^\times)$ as 2-cocycles (or sometimes just cocycles when it is clear from the context) and elements of $\Ho^2(G, \mathbb C^\times)$ the cohomology classes. For an element $\alpha \in Z^2(G, \mathbb C^\times)$, the corresponding element of $\Ho^2(G, \mathbb C^\times)$ will be denoted by $[\alpha]$. For 2-cocycles $\alpha, \beta \in Z^2(G, \mathbb C^\times)$ we say $\alpha$ is cohomologous to $\beta$, whenever $[\alpha] = [\beta]$.

A central extension,
\begin{equation}
\label{equation-stem} 1\to A \to G \to G/A \to 1
\end{equation} is called a {\it stem} extension, if $A \subseteq Z(G) \cap G'$. For a given stem extension (\ref{equation-stem}), the Hochschild-Serre spectral sequence \cite[Theorem 2, p.~129]{HS} for cohomology of groups yields the following exact sequence.
\[1 \rightarrow \Hom( A,\bC^\times) \xrightarrow[]{\tra} \Ho^2(G/A, \bC^\times)  \xrightarrow{\inf} \Ho^2(G, \bC^\times),\]
where  $\tra:\Hom( A,\bC^\times) \to \Ho^2(G/A, \bC^\times) $ given by $f \mapsto [\tra(f)]$, where
$$\tra(f)(\overline{x},\overline{y}) = f(\mu (\overline{x})\mu(\bar{y})\mu(\bar{xy})^{-1}),\,\, \overline{x}, \overline{y} \in G/A, $$ 
for a section $\mu: G/A \rightarrow G$, denotes the transgression homomorphism and the inflation homomorphism, $\inf : \Ho^2(G/A, \bC^\times)   \to  \Ho^2(G, \bC^\times) $ is given by $[\alpha] \mapsto [\inf(\alpha)]$, where $\inf(\alpha)(x,y) = \alpha(xA,yA)$. For groups $H_{2n+1}^t(R)$, We have the following stem extension, 
\begin{equation}
\label{eq:Heisenberg-stem}
1 \rightarrow tR \xrightarrow[]{f} H_{2n+1}^t(R) \xrightarrow[]{g} R^{2n} \oplus R/tR \rightarrow 1, 
\end{equation}
given by 
\[
\begin{array}{c}
f(tr) \mapsto (tr,\underbrace{0, 0, \cdots ,0}_{2n\text{-times}}  ) \\
 g(a, b_1, \ldots, b_n, c_1, \ldots, c_n) = (a \bmod(tR),  b_1, \ldots, b_n, c_1, \ldots, c_n). 
\end{array}
\]

 Let $\alpha \in Z^2(G_1\times G_2, \bC^\times)$. Recall that 
 \begin{align}
\label{eq:direct product}
 \Ho^2(G_1\times G_2, \bC^\times) \cong^\theta   \Ho^2(G_1, \bC^\times) \times \Ho^2(G_2, \bC^\times) \times \Hom(G_1/G'_1 \otimes G_2/G'_2, \bC^\times)
\end{align}
is an isomorphism  defined by 
 $$\theta([\alpha]) = (\res^{G_1\times G_2}_{G_1}([\alpha]), \res^{G_1\times G_2}_{G_2}([\alpha]), \nu),$$ where  $\nu: \Ho^2(G,\bC^\times) \to \Hom(H \otimes K, \bC^\times)$ is a homomorphism given by
   $\nu([\alpha])(\tilde{g}_1 \otimes \tilde{g}_2)=\alpha(g_1,g_2) \alpha(g_2,g_1)^{-1}$, for $\tilde{g}_1 = g_1G_1'$ and $\tilde{g}_2 = g_2G_2'$.
We will use this result without explicitly referring to it.

Now, we recall the definition of the central product of groups. A group $G$ is called a central product of its two  normal subgroups $H$ and $K$ amalgamating $A$ if $G=HK$ with $A=H \cap K$  and $[H,K]=1$.

\begin{thm}$($\cite[Theorem A and Theorem 3.6]{HVY}$)$\label{C1}
Let $G$ be a central product of two normal subgroups $H$ and $K$  amalgamating $A=H \cap K$. Set $Z=H'\cap K'$. 
\begin{enumerate}[label=(\roman*)] 
\item Then the inflation map  $\inf:\Ho^2(G/Z, \bC^\times) \to  \Ho^2(G, \bC^\times)$  is surjective and $$\Ho^2(G, \bC^\times) \cong \Ho^2(G/Z, \bC^\times)/N,$$
where $N \cong \Hom (Z,\bC^\times)$.

\item  The subgroup $\Hom(Z,\bC^\times)$ embeds in $\Ho^2(H/A, \bC^\times)/L \oplus \Ho^2(K/A, \bC^\times)/M$ via $\tra:\Hom(Z,\bC^\times)\to  \Ho^2(G/Z, \bC^\times)$, where $L \cong \Hom \big((A\cap H')/Z, \bC^\times \big)$,  $M \cong \Hom \big((A\cap K')/Z, \bC^\times \big)$.
\end{enumerate}
\end{thm}

\begin{lemma}\label{l1} Let $r \in \mathbb N \cup \{0\}$ and $t$ divides $r$.
\begin{enumerate}[label=(\roman*)]
\item  $\Ho^2(\bZ/t\bZ \oplus (\bZ/r\bZ)^k,\bC^\times) \cong (\bZ/t\bZ)^k \oplus (\bZ/r\bZ)^{\frac{k(k-1)}{2}}$. Further, any $\alpha  \in Z^2(\bZ/t\bZ \oplus (\bZ/r\bZ)^k,\bC^\times)$ with $k \geq 2$ satisfies $[\alpha] = [\nu]$ for $\nu  \in Z^2(\bZ/t\bZ \oplus (\bZ/r\bZ)^k,\bC^\times)$ such that 
\[\nu\big((m_1,m_2,\ldots , m_k,m_{k+1}),(n_1,n_2,\ldots ,n_k, n_{k+1})\big)= \prod_{1 \leq i < j \leq k+1}\mu_{i,j}^{n_im_j},
\]
for some $\mu_{i,j} \in \bC^\times$ satisfying  $\mu_{i,j}^r=1$  for $2 \leq i < j \leq k+1$ and $\mu_{1,l}^t=1$  for $2 \leq  l \leq k+1$.

\item 
Any $\alpha \in Z^2(H_3^t(\bZ/r\bZ),\bC^\times)$  satisfies $[\alpha] = [\sigma]$  for 
$\sigma \in Z^2(H_3^t(\bZ/r\bZ),\bC^\times)$ such that for $x = (m_1,n_1,p_1)$ and $y =  (m_2, n_2, p_2)$ we have, 
\[
\sigma(x, y) = \begin{cases} 
	\lambda^{(m_2p_1+  tn_2\frac{p_1(p_1-1)}{2}) }\mu^{(n_1m_2+tp_1\frac{n_2(n_2-1)}{2} + tp_1n_1n_2)}, & r = 0, \\
	\lambda^{(m_2p_1+  tn_2\frac{p_1(p_1-1)}{2} )}\mu^{(n_1m_2+tp_1\frac{n_2(n_2-1)}{2} + tp_1n_1n_2)}\delta^{(p_1n_2)},  & r \in \mathbb N, 
	\end{cases}
\]
for some $\lambda,\mu,\delta \in \bC^\times$ such that $\lambda^r=\mu^r=\delta^t=1$. 
\end{enumerate}
\end{lemma}
\begin{proof}
(i)  Schur multiplier of finitely generated  abelian groups follows from (\ref{eq:direct product}). We use \cite[Theorem 9.4]{MA} for the cocycle description. We obtain that every cocycle of $\bZ/t\bZ \oplus \bZ/r\bZ$ is cohomologous to a cocycle of the form 
\[
\alpha((m_1,m_2),((n_1,n_2))=\sigma_1(m_1,n_1)\sigma_2(m_2,n_2)g(n_1, m_2), 
\]
where $\sigma_1 \in \Ho^2(\bZ/t\bZ ,\bC^\times), \sigma_2 \in \Ho^2(\bZ/r\bZ ,\bC^\times)$ and $g: \bZ/t\bZ \oplus \bZ/r\bZ \to \bC^\times$ is a map such that $g(n_1,m_2)=g(1,1)^{n_1m_2}=\mu_{1,2}^{n_1m_2}$. The general result follows using induction argument on $k$.

(ii) The proof of this result goes along the same lines as Packer~\cite[Proposion 1.1]{JP1}.  Following the cited proof, we obtain that
every $\alpha \in Z^2(H_3^t(\bZ/r\bZ),\bC^\times)$  is cohomologous to a cocycle of the form 
\[
\begin{array}{l}
\beta((m_1, n_1,p_1), (m_2, n_2, p_2))  = \\   \lambda^{(m_2p_1+  tn_2\frac{p_1(p_1-1)}{2} )}\mu^{(n_1m_2+tp_1\frac{n_2(n_2-1)}{2} + tp_1n_1n_2)}\delta^{(p_1n_2)},
\end{array} 
\] for some $\lambda,\mu,\delta \in \bC^\times$ such that $\lambda^r=\mu^r=\delta^r=1$
First assume that $r=0$. Choose some $\delta_1 \in \bC^\times$ such that $\delta_1^t=\delta$.
Now, define  a function $b: H_3^t(\bZ) \to \bC^\times$ by $b(m,n,p)=\delta_1^m$. Then
\[
b(m_1,n_1,p_1)^{-1}b(m_2,n_2,p_2)^{-1}b(m_1+m_2+tp_1n_2,n_1+n_2,p_1+p_2)=\delta^{p_1n_2}
\]
 is a coboundary. Hence, every cocycle $\alpha \in Z^2(H_3^t(\bZ),\bC^\times)$  is cohomologous to a cocycle of the form 
\[\sigma((m_1,n_1,p_1), (m_2, n_2, p_2))=\lambda^{m_2p_1+  tn_2\frac{p_1(p_1-1)}{2} }\mu^{n_1m_2+tp_1\frac{n_2(n_2-1)}{2} + tp_1n_1n_2},\] for some $\lambda,\mu\in \bC^\times$.

Now, assume $r \in \mathbb N$. If we define a map $b:H_3^t(\bZ/r\bZ) \to \bC^\times$ by $b(m_1,n_1,p_1)=\delta^{m_1}$, then we have $$b(m_1,n_1,p_1)^{-1}b(m_2,n_2,p_2)^{-1}b(m_1+m_2+tp_1n_2,n_1+n_2,p_1+p_2)=\delta^{tp_1n_2},$$ which says  that $\delta^{tp_1n_2}$ is cohomologous to a trivial cocycle.
Then every cocycle $\alpha \in Z^2(H_3^t(\bZ/r\bZ),\bC^\times)$  is cohomologous to a cocycle of the form 
\[\sigma((m_1,n_1,p_1), (m_2, n_2, p_2))=\lambda^{m_2p_1+  tn_2\frac{p_1(p_1-1)}{2} }\mu^{n_1m_2+tp_1\frac{n_2(n_2-1)}{2} + tp_1n_1n_2}\delta^{p_1n_2},\] for some $\lambda,\mu,\delta \in \bC^\times$ such that $\lambda^r=\mu^r=\delta^t=1$.
 
\end{proof}

\begin{corollary} 
	\label{lem: non-trivial-cocycle-abelian}
	Let $r > 1$ and $\mu$ is a primitive $r$-th root of unity. Then 
 $\alpha \in Z^2(\bZ/r\bZ \oplus \bZ /r\bZ)$ defined by $$\alpha((m,n), (m',n' )) = \mu^{nm'},$$ corresponds to a non-trivial element of $\Ho^2( \bZ/r\bZ \oplus \bZ/r\bZ, \mathbb C^\times  )$. 
	
\end{corollary}

\subsection{\textbf{Proof of Theorem \ref{thm1}}}
\begin{proof}
	{\bf (i) Schur multiplier of $H^t_{2n+1}(\bZ/r\bZ)$ for $n > 1$:} 
	Let $G=H^t_{2n+1}(\bZ/r\bZ)$, $ r \in \mathbb{N} \cup \{0\}$ and $n>1$. Then the group $G$ is a central product of $K_1=H^t_{2n-1}(\bZ/r\bZ)$ and $K_2=H^t_3(\bZ/r\bZ)$ amalgamating at $A=Z(G)$. 
	Consider $Z=K'_1\cap K'_2$ which is isomorphic to $t\bZ/r\bZ$. Here $G/Z\cong A/Z \oplus (K_1/A\oplus K_2/A) \cong  \bZ/t\bZ  \oplus   (\bZ/r\bZ)^{2n}$.
	By Theorem \ref{C1}, it follows that the homomorphism $\inf$ of the following exact sequence is surjective. 
	\[
	1 \to \Hom(Z, \bC^\times) \xrightarrow{\tra}\Ho^2( G/Z, \bC^\times) \xrightarrow{\inf}  \Ho^2( G, \bC^\times).
	\]
	Also, $\Hom (t\bZ/r\bZ, \bC^\times)$ embeds in $\Ho^2( K_1/A, \bC^\times) \oplus \Ho^2 (K_2/A, \bC^\times)$ via $\tra$ homomorphism.  
	Hence,  
	\begin{eqnarray}
	\Ho^2( G, \bC^\times) \cong &  \frac{\Ho^2( K_1/A, \bC^\times)\times \Ho^2( K_2/A, \bC^\times)}{ \Hom(Z, \bC^\times)}\times \Hom((\bZ/r\bZ)^{4n-4},\bC^\times) \times (\bZ/t\bZ)^{2n}\nonumber\\
	\cong &  \frac{\Hom((\bZ/r\bZ)^{2n^2-5n+4},\bC^\times) }{\Hom(t\bZ/r\bZ, \bC^\times)}\times \Hom((\bZ/r\bZ)^{4n-4},\bC^\times) \times (\bZ/t\bZ)^{2n}\nonumber\\
	 \cong &  \frac{\Hom((\bZ/r\bZ)^{2n^2-n},\bC^\times) }{\Hom(t\bZ/r\bZ, \bC^\times)} \times (\bZ/t\bZ)^{2n}\label{eq1}.
	\end{eqnarray}

Here the map $\inf: \Ho^2( G/Z, \bC^\times) \to  \Ho^2( G, \bC^\times)$ is surjective,  so every cocycle of $ Z^2(H_{2n+1}^t(\bZ/r\bZ),\bC^\times)$  is cohomologous to a cocycle of the form 
	\[
	\beta((l_1,m_1,\ldots m_{2n}), (l_1',m_1',\ldots m_{2n}'))=\prod_{1\leq i < j\leq 2n}{\mu_{i,j}}^{m_i' m_j}\prod_{k=1}^{2n}{\mu_k}^{l_1'm_k},
	\] for some $\mu_{i,j},\mu_k \in \bC^\times$ and $\mu_k^t=1$ for $1 \leq k \leq 2n$, follows  from Lemma \ref{l1}$(i)$.

	If $r=0$, then $\mu_{i,j}\in \bC^\times$ and $\mu_k^t=1$ for $1 \leq i < j \leq 2n$, $1 \leq k \leq 2n$.  
	Define a map $b:H_{2n+1}^t(\bZ) \to \bC^\times$ such that $b(l_1,m_1,\ldots m_{2n})=(\delta^{1/t})^{l_1}$ for  $\delta \in \bC^\times$. By using $b$, we obtain that $\delta^{(\sum_{1\leq i \leq n}m_i'm_{n+i})}$ is  cohomologous to a trivial cocycle.  
	Therefore, up to cohomologous  we can choose $(\mu_{i,n+i})_{1 \leq i \leq n} \in (\bC^\times)^n/\langle( \delta, \delta,  \delta, \cdots, \delta) \mid \delta  \in \bC^\times \rangle$ which is isomorphic to $(\bC^\times)^{n-1}$.  As by (\ref{eq1}), $(\bC^\times)^{2n^2-n-1} \times (\bZ/t\bZ)^{2n}$ embeds in  $\Ho^2(H_{2n+1}^t(\bZ), \bC^\times)$, hence
	\[
	\Ho^2(H_{2n+1}^t(\bZ), \bC^\times) \cong (\bC^\times)^{2n^2-n-1} \times (\bZ/t\bZ)^{2n}.
	\]
	If $r \in \mathbb N$, then $\mu_{i,j}^r=1$ for $1 \leq i < j \leq 2n$ and $\mu_k^t=1$ for $1 \leq k \leq 2n$. 
	We observe that $x^{(t\sum_{1\leq i \leq n}m_i'm_{n+i})}$ is  cohomologous to a trivial cocycle,  by using the map $b:H_{2n+1}^t(\bZ/r\bZ) \to \bC^\times$ such that $b(l_1,m_1,\ldots m_{2n})=x^{l_1}$, for  $x \in \bC^\times, x^r=1$.
	So, up to cohomologous,  we can choose $(\mu_{i,n+i})_{1 \leq i \leq n} \in (\bZ/r\bZ)^n/\langle (x^t, x^t, x^t, \cdots, x^t)\mid x\in \bZ/r\bZ \rangle \cong (\bZ/r\bZ)^{n-1} \times \bZ/t\bZ$.  Therefore,  by (\ref{eq1}),
	\[
	\Ho^2(H_{2n+1}^t(\bZ/r\bZ), \bC^\times) \cong (\bZ/r\bZ)^{2n^2-n-1} \times (\bZ/t\bZ)^{2n+1}.
	\]
{\bf (ii) Schur multiplier of $H_3^t(\bZ/r\bZ)$:}  
The group $G=H_3^t(\bZ/r\bZ)$ is the semi direct product of normal subgroup $N=\langle (m,n)\rangle \cong \bZ/r\bZ \oplus \bZ/r\bZ$ and a subgroup $T=\langle p \rangle \cong \bZ/r\bZ$, where the action of $T$ on $N$ is defined by $p.(m,n)=(m+tpn, n)$.
Here $T$ act on $\Hom(N, \bC^\times)$ by $(x . f)(n)=f(x.n)$ for $f \in \Hom(N, \bC^\times), n \in N, x \in T$.
Then
$$\Ho^1(T, \Hom(N, \bC^\times))=\frac{Z^1(T, \Hom(N, \bC^\times))}{B^1(T, \Hom(N, \bC^\times))},$$ where 
$$Z^1(T, \Hom(N, \bC^\times))=\{f: T \to \Hom(N, \bC^\times) \mid   f(xy)=(x.f(y)) f(x) \forall x,y \in T \}$$ and
$B^1(T, \Hom(N, \bC^\times))$ consists of 
$f \in Z^1(T, \Hom(N, \bC^\times))$ such that there exists $g \in \Hom(N, \bC^\times)$ satisfying $f(x)=(x.g)g^{-1}$ for all $x \in T$. 

Given  $\alpha \in Z^2(N, \bC^\times)$, let $\alpha^x\in  Z^2(N, \bC^\times)$ be defined by   $\alpha^x(n,n')=\alpha(x.n,x.n')$ for $x \in T$ and $n,n' \in N$.  
Let $\Ho^2(N, \bC^\times)^T$ denote the $T$-stable subgroup of $\Ho^2(N, \bC^\times)$, i.e.,
$$\Ho^2(N, \bC^\times)^T=\{[\alpha] \in \Ho^2(N, \bC^\times) \mid [\alpha^x]=[\alpha]~ \forall~ x \in T\}.$$
We have the following exact sequence. 
\[
1 \to \Ho^1(T, \Hom(N, \bC^\times)) \xrightarrow{\psi} \Ho^2( G, \bC^\times) \xrightarrow{\res} \Ho^2(N, \bC^\times)^T,
\]
which follows from  \cite[Theorem 2.2.5]{GK} and \cite[Corollary 2.5]{JP}  for the finite and infinite discrete cases respectively.  
Here the map $\psi$ is defined by 
\[
\psi([\chi])((m_1, n_1,p_1),(m_2,n_2,p_2))=\chi(p_1)(m_2,n_2), 
\]
for $\chi \in \Ho^1(T, \Hom(N,\bC^\times))$.
Since, by Corollary \ref{lem: non-trivial-cocycle-abelian}, 
every cocycle $\alpha \in Z^2(N,\bC^\times)$  is cohomologous to a cocycle of the form  
$\alpha((m_1,n_1), (m_2, n_2)) = \mu^{n_1m_2}$,
so for $p \in T$, we have
$$\alpha^{p}((m_1,n_1), (m_2, n_2))=\alpha((m_1+tpn_1,n_1), (m_2+tpn_2, n_2))=\mu^{n_1m_2+tpn_1n_2}.$$
Then $[\alpha^p]=[\alpha]$ as 
$$\alpha{\alpha^{p}}^{-1}((m_1,n_1), (m_2,n_2))=b(m_1,n_1)b(m_2,n_2)b(m_1+m_2,n_1+n_2)^{-1},$$ where $b:N \to \bC^\times$ defined by $b(m,n)=\mu^{tpn^2/2}$ (as $r$ is odd). 
Hence,
$$\Ho^2(N,\bC^\times)^T=\Ho^2(N,\bC^\times).$$
Now, we define a map $\phi:\Ho^2(N,\bC^\times) \to \Ho^2(G,\bC^\times)$ given by $[\alpha] \mapsto [\phi[\alpha]]$, where 
\[
\phi([\alpha])((m_1,n_1,p_1),(m_2,n_2,p_2)) =\mu^{n_1m_2+tp_1\frac{n_2(n_2-1)}{2} + tp_1n_1n_2}, 
\]
Then the composition map $\res \circ \phi: \Ho^2(N,\bC^\times) \to \Ho^2(N,\bC^\times)$ becomes the identity homomorphism. Hence, $\phi$ is injective and $\res$ is surjective map. 

Thus we have 
\begin{eqnarray}
\label{eq: Schur multiplier}
\Ho^2( H_3^t(\bZ/r\bZ), \bC^\times)
\cong\Ho^1(T, \Hom(N,\bC^\times)) \times \Ho^2(N,\bC^\times). 
\end{eqnarray}
Now onwards, we consider the cases $r=0$  and $r \in \mathbb N$ separately. 

\noindent {\bf Case~1: $\bm{r = 0}$.}  We follow the proof of \cite[Theorem 2.11]{JP}.
We show that $$\Ho^1(T, \Hom(N,\bC^\times)) \cong \bC^\times.$$
Define a map $\tau: Z^1(T, \Hom(N,\bC^\times)) \to (\bC^\times)^2$ by $\tau(\chi)=(\chi(1)(1,0),\chi(1)(0,1))$.
Then $\tau$ is injective.
For $c_1, c_2 \in \bC^\times$, define $\chi(p)(m,n)=c_1^{(mp+  tn\frac{p(p-1)}{2}) }c_2^{pn}$.
By \cite[Lemma 2.7]{JP}, it follows that $\chi \in Z^1(T, \Hom(N,\bC^\times))$ and $\tau(\chi)=(c_1,c_2)$. So, $\tau$ is surjective. Hence, via the isomorphism $\tau$, we have
\[
Z^1(T, \Hom(N,\bC^\times)) \cong (\bC^\times)^2.
\]
Here $B^1(T, \Hom(N,\bC^\times))$ is the set of all $f:T \to \Hom(N,\bC^\times)$ satisfying the following, 
\[
f(p)(m,n)=g(m+tpn, n)g(m,n)^{-1} \text{~for~} g \in \Hom(N,\bC^\times), m,n \in N, p\in T.
\]
Observe that  $\tau(f)=(1,g((1,0)^t))$ and hence,
$\tau(B^1(T, \Hom(N,\bC^\times)))\cong \bC^\times$. Thus it follows that 
\[
\Ho^1(T, \Hom(N,\bC^\times)) \cong \bC^\times.
\]
Hence, by (\ref{eq: Schur multiplier}),
$$\Ho^2( H_3^t(\bZ), \bC^\times)
 \cong (\bC^\times)^2.$$

\noindent {\bf Case~2: $\bm{r \in \mathbb N}$.} 
For this case, our claim is 
$$\Ho^1(T, \Hom(N, \bC^\times)) \cong  \bZ/r\bZ \oplus \bZ/t\bZ.$$
Let $\zeta$ be a primitive $r$-th root of unity and $\Hom(N,\bC^\times) \cong \langle \phi_1,\phi_2 \rangle$ where $\phi_1:N \rightarrow \mathbb{C}^\times $ is defined by $\phi_1(1,0)=\zeta, \phi_1(0,1)=1$ and $\phi_2(1,0)=1, \phi_2(0,1)=\zeta$. 
Now, $T$ acting on $\Hom(N,\mathbb{C}^\times)$ by $^{p}{\phi_1}(1,0)=\phi_1(1,0)$ and $^{p}{\phi_1}(0,1)=\phi_1(tp,1)=\zeta^{pt}$. 
So, $^{p}{\phi_1}=\phi_1\phi_2^{pt}$. Similarly it is easy to see that $^{p}{\phi_2}=\phi_2$. 
Now, define a  map $Norm:\Hom(N,\mathbb{C}^\times)\rightarrow \Hom(N,\mathbb{C}^\times)$ 
by
\[
Norm(\phi)=\prod_{p \in T} {^{p}{\phi}}.
\]
Consider another map $h: \Hom(N,\mathbb{C}^\times)\rightarrow \Hom(N,\mathbb{C}^\times)$  defined 
by $h(\phi)={^{p}{\phi}}{ \phi}^{-1}$, where $p$ is a generator of $T$.  It is a well known result that $\Ho^1(T, \Hom(N,\bC^\times) \cong \frac{\ker(Norm)}{image (h)} $ (see step 3 in the proof of Theorem 5.4 of \cite{H}). 
Since $r$ is odd, it is easy to check that $Norm(\phi_1)=1$ and $Norm(\phi_2)=1$.
Therefore, $\ker(Norm)=\langle \phi_1,\phi_2 \rangle$ and image of $h$ is $<\phi_2^t>$. Therefore, $\Ho^1(T, \Hom(N, \bC^\times) \cong  \bZ/r\bZ \oplus \bZ/t\bZ$.
Thus by (\ref{eq: Schur multiplier}),
$$\Ho^2( H_3^t(\bZ/r\bZ), \bC^\times)\cong(\bZ/r\bZ)^2 \times \bZ/t\bZ.$$

\end{proof}

\section{Projective representations  of $H^t_{2n+1}(R)$}\label{se2}
In this section, we first recall some basic definitions and results regarding projective representations of a group and then prove Theorems~\ref{thm4},\ref{thm2}, and \ref{thm3}.

Let $V$ be a complex vector space. A projective representation of a group $G$ is a homomorphism of $G$ into the projective general linear group, $\PGL(V) = \GL(V)/Z(V)$. Equivalently, a projective representation is a map $\rho: G \rightarrow \GL(V)$ such that 
\[
\rho(x) \rho(y) = \alpha(x, y) \rho(xy), \,\, \forall x, y \in G, 
\]
for suitable scalars $\alpha(x, y) \in \mathbb C^\times$. By the associativity of  $\GL(V)$, the map $(x,y) \mapsto \alpha(x, y)$ gives a 2-cocycle of $G$, i.e., an element of $Z^2(G, \mathbb C^\times)$. We denote this cocycle by $\alpha$ itself and say $\rho$ is an $\alpha$-representation. Two projective representations $\rho_1: G \rightarrow \GL(V)$ and $\rho_2: G \rightarrow \GL(W)$ are called projectively equivalent if there is an invertible $T\in  \mathrm{Hom}( V, W)$ and a map $b: G \rightarrow \mathbb C^\times$ such that
\[
b(g)T \rho_1(g) T^{-1} = \rho_2(g) ~\forall ~ g \in G.
\]
 Equivalent projective representations are said to have equivalent 2-cocycles. Thus two cocycles $\alpha, \alpha' :G \times G \rightarrow \mathbb C^\times$ are equivalent if there exists a map $b : G \rightarrow \mathbb C^\times$ such that $\alpha(x, y) = \frac{b(x) b(y)}{b(xy)} \alpha'(x, y)$ for all $x,y \in G$. In terms of Schur multiplier, this means that the representations $\rho$ and $\rho'$ are equivalent implies that their cocycles $\alpha$ and $\alpha'$ are cohomologous, i.e., $[\alpha] = [\alpha']$ in $\Ho^2(G, \mathbb C^\times)$. It is to be noted that to determine all projective representations of $G$ up to equivalence, it is enough to determine projectively inequivalent $\alpha$-representations of $G$ for a set of all 2-cocycle  representatives of elements of $\Ho^2(G, \mathbb C^\times)$. We further note that two projectively equivalent $\alpha$-representations $(\rho_1, V)$ and $(\rho_2, W)$ are called linearly inequivalent if $b(g) = 1$ for all $g \in G$. Any $\alpha$-representation $\rho$ of $G$ such that $\alpha$ is cohomologous to trivial 2-cocycle, will be called equivalent to an ordinary representation of $G$.

The set of all inequivalent irreducible ordinary representations of a group $G$ will be denoted by $\mathrm{Irr}(G)$.  
Let $\irr^{\alpha}(G)$ be the set of complex linearly inequivalent irreducible representations of $G$ corresponding to a 2-cocycle $\alpha$.  
We can further assume that $\alpha$ is normalized cocycle, i.e.,  $\alpha  \in Z^2(G, \mathbb C^\times)$ satisfies
\begin{eqnarray}
\label{eq: cocycle condition}
\alpha(g, 1) = \alpha (1, g) = 1,\;\; \forall g \in G.
\end{eqnarray} 
Throughout this section, we assume that the cocycle representative of $[\alpha]$  with which we work, satisfies (\ref{eq: cocycle condition}). Next, we recall the definition of a representation group (also called a covering group) of a group $G$ from \cite[Page 23]{IS4}. 

%\begin{definition}
%A group $G^*$ is called a \emph{representation group} of  the group $G$ if the following conditions are satisfied:
%\begin{enumerate}[label=(\roman*)]
%\item There exists a central extension $1 \rightarrow A \rightarrow  G^* \rightarrow G \rightarrow 1$ such that $\Hom(A,\bC^\times)  \cong \Ho^2(G, \mathbb C^\times)$.  
%\item For every  projective representation $\rho$ of $G$, there exists an ordinary representation $\tilde{\rho}$ of $G^*$ such that $\rho(g)=\tilde{\rho}(s(g))$ for all $g \in G$ and for some section $s: G\rightarrow G^*.$
%\end{enumerate}
%\end{definition}
\begin{definition}[Representation group of $G$]
	\label{defn:representation-group}
	A group $G^*$ is called a  \emph{representation group} of $G,$ if there is a central extension
	$$1 \rightarrow A \rightarrow   G^* \rightarrow G \to 1$$ such that
	corresponding transgression map $$\tra: \Hom(A,\bC^\times) \to  \Ho^2(G,\bC^\times)$$ is an isomorphism.
\end{definition}
In \cite{IS4}, Schur proved that the representation group of a finite group always exists. However, for infinite groups, the parallel result is not known; see \cite{hatui2020projective} for related results.
The next result relates the projective representations of a group $G$ and its certain quotient group.

\begin{thm}\label{inf}
Let $A$ be a subgroup of a finitely generated group $G$ such that $A \subseteq G'\cap Z(G)$ and, $[\alpha]\in  \Ho^2(G,\bC^\times) $ be in the image of $\inf:\Ho^2(G/A,\bC^\times) \to \Ho^2(G,\bC^\times)$. Then 
$\bigcup_{\{[\beta] \in \Ho^2(G/A,\bC^\times)\mid \inf([\beta])=[\alpha]\}}\irr^{\beta}(G/A)$ and $\irr^{\alpha}(G)$ are in bijective correspondence via inflation.
\end{thm}
\begin{proof}  We have the following exact sequence 
\[
1 \to \Hom(A,\bC^\times) \stackrel{\tra}\to \Ho^2(G/A, \bC^\times)  \stackrel{\inf}\to \Ho^2(G, \bC^\times).
\]
Fix a $[\beta] \in \Ho^2(G/A, \bC^\times)$ such that $\inf([\beta])=[\alpha]$.  Due to the exactness of the  above sequence, the set $\bigcup_{\chi \in \Hom(A,\bC^\times)}[\beta] \tra(\chi)$ consists of all distinct elements of $\Ho^2(G/A, \bC^\times)$ that map to $[\alpha]$ via inf.

Let $\rho:G \to \GL(V)$ be an irreducible $\alpha$-representation of $G$. Then  there exists a representative of $[\beta]$, denoted by $\beta$ itself, such that $\alpha(g,h)=\beta(gA,hA)$ for all $g,h \in G$.  
Therefore, for all $a \in A$ and $g\in G$, we have  $\alpha(g,a)=\alpha(a,g) = 1$.  Hence, 
$$\rho(g)\rho(a)=\rho(a)\rho(g),\,\, \forall a \in A, g\in G.$$
Since every irreducible representation in our case is countable dimensional,  by Schur's lemma (due to Dixmier for countable dimensional complex representations), for all $a \in A$, $\rho(a)$ is a scalar multiple of identity. Further $\alpha(a,a') = 1$ for all $a, a' \in A$, so  $\rho|_A$ is a homomorphism on $A$.
Let $\mu :G/A \rightarrow G$ be a section of $G/A$ in $G$ such that $gA=\mu(gA)A $ for all $g \in G$. 
Every element $g \in G$ can be written uniquely $g=a_g\mu(gA)$ for some $a_g \in A$.
Note that $\tra(\rho|_A)(gA, hA)=\rho(\mu(gA)\mu(hA)\mu(ghA)^{-1})$. 
Now, define $\tilde{\rho}:G/A \to \GL(V)$ by $\tilde{\rho}(gA)=\rho(\mu(gA))$.
Then
\begin{equation} 
\left.\begin{aligned}
 \tilde{\rho}(gA)\tilde{\rho}(hA)\tilde{\rho}(ghA)^{-1} &=  \rho(\mu(gA))\rho(\mu(hA))\rho(\mu(ghA))^{-1}\\
& =\beta(gA,hA)\rho(\mu(gA)\mu(hA))\rho(\mu(ghA))^{-1}\\
& =\beta(gA,hA)\rho(\mu(gA)\mu(hA)\mu(ghA)^{-1}\mu(ghA))\rho(\mu(ghA))^{-1}\\
& =(\beta\tra(\rho|_A))(gA,hA)\alpha^{-1}(\mu(gA)\mu(hA)\mu(ghA)^{-1},\mu(ghA))\\
& =(\beta\tra(\rho|_A))(gA,hA),
\end{aligned}\right.
\end{equation}
where $\alpha^{-1}(\mu(gA)\mu(hA)\mu(ghA)^{-1},\mu(ghA))=1$ as $\mu(gA)\mu(hA)\mu(ghA)^{-1} \in A$.
Thus $\tilde{\rho}$ is $\beta'$-representation of $G/A$ such that $[\beta']=[\beta] [\tra(\rho|_A)]$ and $\inf([\beta'])=[\alpha]$. 
Since $\rho$ is irreducible representation and $\rho(a)$ is a scalar multiple of identity for $a\in A$,  $\tilde{\rho}$ is also an irreducible representation.

Define a map 
$$\phi : \irr^{\alpha}(G) \longrightarrow \bigcup_{\{[\beta] \in \Ho^2(G/A,\bC^\times)\mid \inf([\beta])=[\alpha]\}}\irr^{\beta}(G/A)$$  by  $\phi(\rho)=\tilde{\rho}$. It is easy to see that $\phi$ is  a well defined map.  Next, we prove that $\phi$ is injective. Suppose $\rho,\rho' \in \irr^\alpha(G)$ and  $\phi(\rho)=\tilde{\rho},  \phi(\rho')=\tilde{\rho}'$ such that $\tilde{\rho}$ and $\tilde{\rho}'$ are linearly equivalent, i.e., $\tilde{\rho}'(gA)=T\tilde{\rho}(gA)T^{-1}$ for all $g \in G$ and for some $T \in \GL(V)$. Since  $\tilde{\rho}$ and $\tilde{\rho}'$ are $\beta \tra(\rho|_A)$ and $\beta \tra(\rho'|_A)$-representations of $G/A$  respectively,  $\tra(\rho|_A)=\tra(\rho'|_A)$. But $\tra$ is injective, so $\rho|_A=\rho'|_A$. Now it is easy to check that $\rho'(g)=T\rho(g)T^{-1}$ for $g \in G$. Hence, $\phi$ is injective.
It remains to  show that $\phi$ is surjective. Let  $\tilde{\rho}:G/A \to \PGL(V)$ be an irreducible $\beta_1$-projective representation such that $\inf(\beta_1)=\alpha$.  Define $\rho:G \to \PGL(V)$  via inflation, i.e., $\rho(g)=\tilde{\rho}(gA)$. Then $\rho$ is an irreducible $\alpha$-representation of $G$ and $\phi(\rho)=\tilde{\rho}$.
\end{proof}

\begin{corollary}\label{rep_group}
Let $A$ be a central subgroup of a finitely generated group $G^*$ such that $G^*$ is a representation group of $G = G^*/A$. Then there is a bijection between
the sets  $\cup_{[\alpha]\in \Ho^2(G, \mathbb C^\times)} \mathrm{Irr}^\alpha(G)$ and $\mathrm{Irr}(G^*)$.
\end{corollary}
\begin{proof}
By the definition of representation group and the exactness of the sequence \[\Hom(G^*,\bC^\times) \xrightarrow[]{\res}  \Hom( A,\bC^\times) \xrightarrow[]{\tra} \Ho^2(G, \bC^\times)  \xrightarrow{\inf} \Ho^2(G^*, \bC^\times),\]
we have $\res: \Hom(G^*, \bC^\times) \rightarrow \Hom( A,\bC^\times)$ is trivial. Hence, $A \subseteq [G^*,G^*]$. Since $\inf$ is a trivial map,  result follows from Theorem \ref{inf}.
\end{proof}

\begin{corollary} \label{HG}
Let $G$ be a central product of its subgroups $H$ and $K$ with $Z=H'\cap K'$. Then every projective representation  of $G$ is obtained from a projective representation of $G/Z$ via inflation. 
\end{corollary} 
\begin{proof}
By Theorem \ref{C1}, it follows that $\inf:\Ho^2(G/Z,\bC^\times)\to \Ho^2(G, \bC^\times)$ is a surjective map.
Therefore, proof follows by Theorem~\ref{inf}.  
\end{proof}

\subsection{\textbf{Proof of Theorem \ref{thm2} and Corollary \ref{thm5} }}
\label{subsec:central-product-proof}

For $n>1$,  $H_{2n+1}^t(R)$ is a central product of $H_{2n-1}^t(R)$ and $H_3^t(R)$.
We obtain a natural homomorphism from $\Ho^2(R^{2n}\oplus R/tR, \bC^\times)$ to $ \Ho^2(H_{2n+1}^t(R), \mathbb C^\times),$ via inflation. Let $[\alpha]$ be a cohomology class of $H_{2n+1}^t(R)$. We obtain the following from Theorem~\ref{C1} and Corollary~\ref{HG}.

	\begin{enumerate}[label=(\roman*)]
		\item The  inflation map from $\Ho^2(R^{2n}\oplus R/tR, \mathbb C^\times)$ to $\Ho^2(H_{2n+1}^t(R), \mathbb C^\times)$ is surjective.  
		\item Every irreducible $\alpha$-representation of $H_{2n+1}^t(R)$ is obtained by composing a irreducible $\beta$-representation of $R^{2n}\oplus R/tR$ for some $\beta \in Z^2(R^{2n}\oplus R/tR, \bC^\times)$ such that $[\alpha] = \mathrm{inf}([\beta])$.
	\end{enumerate}

The proof of Theorem~\ref{thm2} now follows from (ii). Similarly, the group $ES_{2n+1}(p^2)$ is a central product of $ES_{2n-1}(p^2)$ and $ES_3(p^2)$, hence Corollary~\ref{thm5}(ii) again follows from Corollary $\ref{HG}$.

\subsection{Proof of Theorem \ref{thm3}}
\begin{proof}
For finite abelian groups it follows by \cite[Theorem 5.4 in Chapter 3]{GK1} that $ \mathcal{F}_{n}(r,t)$ is a representation group of $(\bZ/r\bZ)^n \oplus \bZ/t\bZ$.
Hence, in the proof below, we assume $r=0$ and $t$ is a positive integer.  
However, we remark that the following proof also works for  $r \in \mathbb N$ and not the same as appeared in  \cite[Theorem 5.4 in Chapter 3]{GK1}.
Consider $Z=\langle z_{ij}, 1\leq i <j\leq n+1 \rangle$, a central subgroup of $\mathcal{F}_{n}(r,t)$.
There exists a central extension
	$$1 \to Z \to \mathcal{F}_{n}(r,t) \xrightarrow{\pi}  \bZ/t\bZ \oplus   (\bZ/r\bZ)^n    \to 1,$$
where $\pi$ is defined by $\pi(\prod_{i=1}^{n+1}x_i^{m_i}\prod_{1\leq i <j\leq n+1 } z_{ij}^{k_{ij}})=(m_1, m_2,\ldots ,m_{n+1})$.	
Then we have the exact sequence
	$$1\to \Hom(Z,\bC^\times) \xrightarrow{\tra} \Ho^2(\mathcal{F}_{n}(r,t)/Z, \bC^\times)  \xrightarrow{\inf} \Ho^2( \mathcal{F}_{n}(r,t), \bC^\times).$$
	We want to show that  inf is a  trivial homomorphism. 
	
	Let $X=\prod_{i=1}^{n+1}x_i^{m_i}\prod_{1\leq i <j\leq n+1 } z_{ij}^{k_{ij}}$ and  $Y=\prod_{i=1}^{n+1}x_i^{m'_i}\prod_{1\leq i <j\leq n+1} z_{ij}^{k'_{ij}}$   be two elements of $ \mathcal{F}_{n}(r,t).$ Then the element $XY$ is of the following form:
	\begin{eqnarray*}
		XY&=&x_1^{m_1}x_2^{m_2}\ldots x_{n+1}^{m_{n+1}}. x_1^{m'_1}x_2^{m'_2}\ldots x_{n+1}^{m'_{n+1}}.\prod_{1\leq i <j\leq n+1 } z_{ij}^{k_{ij}+k'_{ij}}\\
		&=&x_1^{m_1+m'_1}x_2^{m_2+m'_2}\ldots x_{n+1}^{m_{n+1}+m'_{n+1}}.
		\prod_{1\leq i <j\leq n+1} z_{ij}^{k_{ij}+k'_{ij}-m'_im_j}.
	\end{eqnarray*}
	Let $\alpha \in Z^2(\mathcal{F}_{n}(r,t)/Z, \bC^\times) $ and $\inf([\alpha])=[\beta]$. Then by Lemma \ref{l1}, 
	\begin{eqnarray*}
		\beta(X,Y) &=&\alpha(\pi(X),\pi(Y))= \alpha\big((m_1, m_2,\ldots ,m_{n+1}), (m'_1, m'_2,\ldots ,m'_{n+1})\big)\\
		&=& \prod_{1 \leq i < j \leq n+1}\mu_{i,j}^{m'_im_j},
	\end{eqnarray*}
	for some $\mu_{i,j}\in \bC^\times$.
	Define a function $\tau: \mathcal{F}_{n}(r,t) \to \bC^\times$ by 
	$$\tau(x_1^{m_1}x_2^{m_2}\ldots x_r^{m_r}\prod_{1\leq i <j\leq n+1 } z_{ij}^{k_{ij}})=\prod_{1\leq i <j\leq n+1 } \mu_{i,j}^{-k_{ij}}.$$
	Now we have 
	\begin{eqnarray*}
		&&\tau(X)^{-1}\tau(Y)^{-1}\tau(XY)\\
		&&=\prod_{1\leq i <j\leq n+1 } \mu_{i,j}^{k_{ij}} \prod_{1\leq i <j\leq n+1 } \mu_{i,j}^{k'_{ij}}  \prod_{1\leq i <j\leq n+1 } \mu_{i,j}^{-k_{ij}-k'_{ij}+m_i'm_j}\\
		&&=\prod_{1\leq i <j\leq n+1}\mu_{i,j}^{m_i'm_j}=\beta(X,Y).
	\end{eqnarray*}
	Hence, $\beta$ is, in fact, a coboundary, and therefore $\inf$ is trivial.  This along with Theorem \ref{inf} and Lemma \ref{l1} completes the proof.
\end{proof}

\subsection{\textbf{Proof of Theorem \ref{thm4}}}

\begin{proof}

Consider $Z=\langle z_1,z_2, z^r \rangle$ which is a central subgroup of $\widehat{H}(r,t).$
Now consider the central extension 
\[ 1 \to Z \to \widehat{H}(r,t) \xrightarrow{\pi}H_3^{t}(\bZ/r\bZ) \to 1, \]
where $\pi$ is defined by $\pi(z_1^{k_1}z_2^{l_1}z^{m_1}y^{n_1}x^{p_1})=(m_1, n_1,p_1)$.
Then we have the following exact sequence. 
\[
1\to \Hom(Z,\bC^\times) \xrightarrow{\tra} \Ho^2(H_3^t(\bZ/r\bZ), \bC^\times)  \xrightarrow{\inf} \Ho^2( \widehat{H}(r,t),\bC^\times).
\]
We have the following relations in $\widehat{H}(r,t)$.
\begin{eqnarray*}
&&[x^n,y]=[x,y]^n[x,z^t]^{\frac{n(n-1)}{2}}=z^{tn}z_1^{\frac{tn(n-1)}{2}},\\
&&[x,y^n]=[x,y]^n[y,z^t]^{\frac{n(n-1)}{2}}=z^{tn}z_2^{\frac{tn(n-1)}{2}},\\
&&[x^m,y^n]=z^{tmn}z_1^{tn\frac{m(m-1)}{2}}z_2^{tm\frac{n(n-1)}{2}}.
\end{eqnarray*}
Let $X=z_1^{k_1}z_2^{l_1}z^{m_1}y^{n_1}x^{p_1}$ and  $Y=z_1^{k_2}z_2^{l_2}z^{m_2}y^{n_2}x^{p_2}$  be two elements of $\widehat{H}(r,t).$
Then $XY=z_1^{k_1}z_2^{l_1}z^{m_1}y^{n_1}x^{p_1}.z_1^{k_2}z_2^{l_2}z^{m_2}y^{n_2}x^{p_2}  $ has the following expression. 
\begin{eqnarray*}
z_1^{k_1+k_2+m_2p_1+  tn_2\frac{p_1(p_1-1)}{2} }z_2^{l_1+l_2+n_1m_2+tp_1\frac{n_2(n_2-1)}{2} +tp_1n_1n_2 }z^{m_1+m_2+tp_1n_2}y^{n_1+n_2}x^{p_1+p_2}.
\end{eqnarray*}
We first assume that $r \in \mathbb N$. Then by Lemma \ref{l1}$(ii)$, every $\alpha \in Z^2(H_3^t(\bZ/r\bZ),\bC^\times)$ is cohomologous to a cocycle of the form 
\[\alpha((m_1,n_1,p_1), (m_2, n_2, p_2))=\lambda^{m_2p_1+  tn_2\frac{p_1(p_1-1)}{2} }\mu^{n_1m_2+tp_1\frac{n_2(n_2-1)}{2} +tp_1n_1n_2}\delta^{p_1n_2},
\] for some $\lambda,\mu, \delta \in \bC^\times$.
Let $\alpha \in Z^2(H_3(\bZ/r\bZ), \bC^\times)$. Then $\inf([\alpha])=[\beta]$ is given by
\begin{eqnarray*}
\beta(X,Y) &=& \alpha(\pi(X), \pi(Y))\\
&=&\alpha((m_1,n_1,p_1), (m_2, n_2, p_2))\\
&=&\lambda^{m_2p_1+  tn_2\frac{p_1(p_1-1)}{2} }\mu^{n_1m_2+tp_1\frac{n_2(n_2-1)}{2} +tp_1n_1n_2}\delta^{p_1n_2}.
\end{eqnarray*}
Define a function $b: \widehat{H}(r,t) \to \bC^\times$ by $b(z_1^{k_1}z_2^{l_1}z^{m_1}y^{n_1}x^{p_1})=\lambda^{k_1}\mu^{l_1}\delta_1^{m_1}$, where $\delta_1 \in \bC^\times$ such that $\delta_1^{rt}=1$ and $\delta_1^t=\delta$. The existence of such $\delta_1$ follows by $t|r$.
Then we have
\begin{eqnarray*}
&&b(X)^{-1}b(Y)^{-1}b(XY)\\
&&=\lambda^{m_2p_1+ t n_2\frac{p_1(p_1-1)}{2} }\mu^{n_1m_2+tp_1\frac{n_2(n_2-1)}{2} +tp_1n_1n_2}\delta^{p_1n_2}\\
&&=\beta(X,Y).
\end{eqnarray*}
Therefore, $\inf$ is trivial.
By Theorem \ref{inf} and Lemma \ref{l1}, our result follows for $r \in \mathbb{N}$. 
For $r=0$, proof goes on the same lines as above by defining the function $b: \widehat{H}(r,t) \to \bC^\times$ by $b(z_1^{k_1}z_2^{l_1}z^{m_1}y^{n_1}x^{p_1})=\lambda^{k_1}\mu^{l_1}$.

\end{proof}

\section{Ordinary representations of $\widehat{H}(r,t)$ and $\mathcal{F}_n(r,t)$ for $ r \in \mathbb N \cup \{0\}$}\label{se3}

In this section, we discuss methods to obtain the irreducible representations of $\widehat{H}(r,t)$ and $\mathcal{F}_n(r,t)$ for both the finite as well as the discrete case. For this, first we define induction for the discrete case and state some of the required results. Then we prove a general statement that gives a uniform construction of the irreducible representations for $\widehat{H}(r,t)$ and $\mathcal{F}_n(r,t)$. We use the notation $\mathrm{Irr}(G)$ to denote the isomorphism classes of all irreducible ordinary representations of $G$. Let $\mathrm{Irr}^\circ(G) = \{\rho \in \irr(G) \mid \dim(\rho) < \infty    \}$. 

For a  normal subgroup $N$ of $G$ and $\rho \in \mathrm{Irr}^\circ(N)$, the sets  $\{ \delta \in \mathrm{Irr}(G) \mid \langle  \delta|_N,  \rho \rangle \neq 0 \}$  and $\{ \delta \in \mathrm{Irr}^\circ(G) \mid \langle  \delta|_N,  \rho \rangle \neq 0 \}$ are denoted by $\mathrm{Irr}(G \mid \rho)$ and $\mathrm{Irr}^\circ(G \mid \rho)$ respectively. We use the following definition of induced representation for the discrete groups. This is an analogue of compact induction for Lie groups and  has already been explored in literature; see for example Parshin~\cite[Definition 1]{Parshin}. 
\begin{definition}(Induced representation)
	\label{defn:induced}
	Let $H$ be a subgroup of a finitely generated group $G$ and $(\rho, W)$ be a
	representation of $H$. The induced representation
	$(\widetilde{\rho}, \widetilde{W})$ of $\rho$ from $H$ to $G$ has
	representation space $\widetilde{W}$ consisting of functions $f: G
	\rightarrow W$ satisfying the following:
	\begin{enumerate}
		\item $f(hg) = \rho(h) f(g)$ for all $g \in G$ and $h \in H$.
		\item The support of $f$ is contained in a union of finitely many
		right cosets of $H$ in $G$.
	\end{enumerate}
	The homomorphism $\widetilde{\rho}: G \rightarrow
	\mathrm{Aut}(\widetilde{W})$ is given by $\widetilde{\rho}(g)f(x)
	= f(xg)$ for all $x, g \in G$. We denote this induced
	representation by $\ind_H^G(\rho)$.
\end{definition}

We note that it agrees with the usual definition of induction for finite groups. 
We use a few standard properties of the above induction in the next result; see~\cite[Remark~2.6]{NaSi2018} for exact results used.

\begin{proposition}
	\label{proposition: general construction}
	 Let $G$ be a finitely generated discrete group with a normal subgroup $N$ such that $G/ N $ is cyclic. Let $(\rho, V)$ be an irreducible representation of $N$ and let  $I_G(\rho) = \{ g \in G \mid \rho^g \cong \rho\}$ be the inertia group of $\rho$ in $G$. Then the following are true.
	
	\begin{enumerate}
		\item The representation $\rho$ extends to $I_G(\rho)$.
		\item Any $\delta \in \mathrm{Irr}^\circ(I_G(\rho))$  such that $\langle \rho, \delta|_N\rangle \neq 0$ satisfies the following.
		\begin{enumerate} 
		\item $\delta|_N = \rho$. 
		\item  The representation $\ind_{I_G(\rho)}^G(\delta)$ is irreducible. 
		 \end{enumerate} 
		\item For $|G/I_G(\rho) | < \infty$, the sets $\mathrm{Irr}^\circ(I_G(\rho) \mid \rho) $ and $ \mathrm{Irr}^\circ(G \mid \rho)$ are in bijection via  $\delta \mapsto \ind_{I_G(\rho)}^G(\delta)$. 
	\end{enumerate}
	
\end{proposition}

\begin{proof} For the finite group $G$, (1) is well known; see~\cite[Theorem~11.7]{MI}.
	We remark that the proof of the above-cited result also works for infinite cases as long as $\Ho^2(G/N, \mathbb C^\times) = 1$. This fact is well known for discrete cyclic groups.  Therefore, the result follows in this case also.  
	
	For finite groups, both (2) and (3) are consequences of the Clifford theory. So, we only deal with the case of infinite discrete group $G$. Let $(\delta, W)$ be a finite-dimensional representation of $I_G(\rho)$ such that $\langle \rho, \delta|_N  \rangle \neq 0$. Let $y \in G$ such that $I_G(\rho)/N = \langle yN \rangle $. Then we have $V \subseteq W$. For $V = W$, we are done. Otherwise, there exists smallest $t \in \mathbb N$ such that $W = V \oplus V^y \oplus V^{y^2} \oplus \cdots \oplus V^{y^{t-1}}$ and $V^{y^t} = V$. Here we have used the fact that $W$ is finite-dimensional and both $V$ and $W$ are irreducible. Consider a subgroup $S = \langle y^t \rangle$ of $I_G(\rho)$ and its action on the finite-dimensional space $V$ of $N_t = \langle N, S \rangle$ via $\delta$. Then by (1), the representation $\rho$ extends to a representation $\tilde{\rho}$ of $N_t$ such that $\tilde{\rho}|_N = \rho$ and $\langle \tilde{\rho}, \delta|_{N_t}  \rangle \neq 0$.  The group $N_t$ is a finite index subgroup of $I_G(\rho)$.  Therefore, by Frobenius reciprocity, we obtain $\langle \ind_ {N_t}^{I_G(\rho)}(\tilde{\rho}), \delta \rangle \neq 0 $.  We note that $\ind_ {N_t }^{I_G(\rho)}(\tilde{\rho})$ is a finite-dimensional representation. By part (1) we obtain, $$\ind_ {N_t }^{I_G(\rho)}(\tilde{\rho}) \cong \oplus_{\chi \in \widehat{I_G(\rho)/ N_t } } \tilde{\rho}\otimes \chi.$$ Therefore, $\tilde{\rho}\otimes \chi \cong \delta$ for some $\chi \in \widehat{I_G(\rho)/ N_t }$. This implies  $\delta|_{N} = \rho$. 
	Next, we note that
	\[
	\mathrm{End}_G(\ind_{I_G(\rho)}^G(\delta)) \cong \oplus_{g \in G/I_G(\rho)} \mathrm{Hom}_{I_{G}(\rho)}(\delta, \delta^g).
	\]
By definition of $I_G(\rho)$ and the fact that $\delta|_N = \rho$, we have $\mathrm{Hom}_{I_{G}(\rho)}(\delta, \delta^g) \neq 0$ for $g \in G/ I_G(\rho)$ if and only if $g \in I_G(\rho)$.  This implies that $\mathrm{End}_G(\ind_{I_G(\rho)}^G(\delta)) \cong \mathbb C$, that is $\ind_{I_G(\rho)}^G(\delta)$ is Schur irreducible. By \cite[Theorem~3.1]{NaSi2018}, we obtain that $\ind_{I_G(\rho)}^G(\delta)$ is irreducible. Finally,  (3) follows by the definition of $I_G(\rho)$, $\delta|_{N} = \rho$ and the fact that,
\[
\ind_{I_G(\rho)}^G(\delta) |_{I_G(\rho)} \cong \oplus_{g \in G/I_G(\rho)} \delta^g.
\]
		
	\end{proof}

\subsection{Construction for two-step nilpotent groups} 
\label{subsection:two-step construction}
In this section, we outline a well-known method to construct all finite-dimensional irreducible representations of a two-step nilpotent group $G$.

\begin{enumerate}
	\item Let $\chi: Z(G) \rightarrow \mathbb C^\times$ be a one dimensional character of $Z(G)$ such that $\chi|_{G'}$ is of finite order. Define the bilinear form, 
	$$\beta_\chi: G/Z(G) \times G/Z(G) \rightarrow \mathbb C^\times;\,\, \beta_\chi(x Z(G), yZ(G)) = \chi([x,y])$$
	\item Let $R_\chi = \{ g \in G \mid \beta_\chi(g,g') = 1\,\, \forall \,\, g' \in G    \}$. Then the character $\chi$ extends to $R_\chi$. 
	\item For every $\tilde{\chi} \in \mathrm{Irr}^\circ(R_\chi \mid \chi)$, there exists a unique irreducible representation, denoted $\rho_{\tilde{\chi}} \in \mathrm{Irr}(G \mid \chi)$. 
	\item By \cite[Theorem~1.3]{NaSi2018},  $\rho_{\tilde{\chi}} \in \mathrm{Irr}^\circ(G \mid \chi)$ because $\chi|_{G'}$ has finite order. Furthermore, we have $\dim(\rho_{\tilde{\chi}} ) = \sqrt{|G|/|R_\chi|}$.
	\item The map $\tilde \chi \mapsto \rho_{\tilde \chi}$ gives a bijection in the sets $\mathrm{Irr}^\circ(R_\chi \mid \chi )$ and $\mathrm{Irr}^\circ(G \mid \chi)$. 
\end{enumerate}
The benefit of this method over Mackey Theory for two-step nilpotent groups lies in the fact that many properties about irreducible representations can be easily deduced from this construction. For example, every finite-dimensional irreducible representation of a two-step nilpotent group is monomial follows directly from the above construction. Also, determining the dimensions of all finite-dimensional irreducible representations is easier in this case. For example, the construction implies that all $\rho \in \mathrm{Irr}^\circ(G \mid \chi)$ satisfy $\dim(\rho_{\tilde{\chi}} ) = \sqrt{|G|/|R_\chi|}$.
% By \cite[Theorem~1.3]{NaSi2018}, an irreducible representation $\rho$ of $G$ is finite dimensional if and only if  $\rho = \rho_{\tilde{\chi}}$ where $\rho_{\tilde{\chi}}$ is as above. 

\subsection{ Irreducible representations of  $\widehat{H}(r,t)$ and $\mathcal{F}_n(r,t)$. } 
The group $\mathcal{F}_n(r,t)$ is a two-step nilpotent group. So, its ordinary representations can be directly obtained from its central characters as in Section~\ref{subsection:two-step construction}. However,  below, by using Proposition~\ref{proposition: general construction} and  Section~\ref{subsection:two-step construction}, we indicate a method  that works for both $\mathcal{F}_n(r,t)$ and $\widehat{H}(r,t)$. 

Consider the subgroups $N_H = \langle x, z, z_1, z_2 \rangle $ and $N_F = \langle x_k , z_{ij}, z_k \mid 1 \leq k \leq n, 1 \leq i < j \leq n \rangle$ of $\widehat{H}(r,t)$ and  $\mathcal{F}_n(r,t)$ respectively. Then $N_H$ and $N_F$ are normal subgroups of $\widehat{H}(r,t)$ and $\mathcal{F}_n(r,t)$  such that $\widehat{H}(r,t)/N_H$ and $\mathcal{F}_n(r,t)/N_F$ are cyclic. We note that both $N_F$ and $N_H$ are two-step nilpotent groups. Therefore, irreducible representations of these are obtained from  one dimensional representation of the radical of each central character as described in Section~\ref{subsection:two-step construction}. So, it remains to determine the inertia group of these representations of $N_H$ and $N_F$ in $\widehat{H}(r,t)$ and  $\mathcal{F}_n(r,t)$ respectively and then the construction is obtained by Proposition~\ref{proposition: general construction}. 

\section*{Acknowledgements} 
The authors are grateful to E. K. Narayanan for very insightful discussions regarding this project and for his constant encouragement. They are also indebted to Yuval Ginosar for carefully reading this article and for providing many helpful comments. 
The authors also thank the referee for careful reading and comments.
SH acknowledges the NBHM grant (0204/52/2019/R$\&$D-II/333) and PS acknowledges the SERB MATRICS grant (MTR/2018/000094) respectively. Both SH and PS thank UGC CAS-II grant (Grant No. F.510/25/CAS-II/2018(SAP-I)) at the Indian Institute of Science, Bangalore.

\end{document}